\documentclass[12pt,a4paper]{article}
\usepackage{blindtext}
\usepackage{enumerate, tikzsymbols}
\usepackage{amssymb}
\usepackage{amsmath}
\usepackage{latexsym}
\usepackage{amsthm}
\usepackage{xypic}
\usepackage{authblk, hyperref}

\usepackage{amsmath}

\newtheorem{theorem}{Theorem}

\newtheorem{proposition}{Proposition}
\newtheorem{cor}{Corollary}
\theoremstyle{definition}

\theoremstyle{remark}
\newtheorem{remark}[theorem]{Remark}

\theoremstyle{claim}
\newtheorem{claim}[]{Claim}

\theoremstyle{question}

\numberwithin{equation}{section}

\usepackage[english]{babel}

\newcommand{\Z}{\mathbb{Z}}
\newcommand{\Q}{\mathbb{Q}}
\newcommand{\Qbar}{\bar{\mathbb{Q}}}
\newcommand{\Gal}{\mathrm{Gal}}

\newcommand{\kAtor}{k(A_{\mathrm{tors}})}
\newcommand{\KAtor}{K(A_{\mathrm{tors}})}
\newcommand{\GL}{\mathrm{GL}}

\usepackage{etoolbox}
\usepackage{hyperref}

\makeatletter
\newcounter{author}
\renewcommand*\author[1]{%
  \stepcounter{author}%
  \ifnum\c@author=1
    \gdef\@author{#1}%
  \else
    \xdef\@author{\unexpanded\expandafter{\@author\and#1}}%
  \fi
  \csgdef{author@\the\c@author}{#1}}
\newcommand*\email[1]{%
  \csgdef{email@\the\c@author}{#1}}
\newcommand*\address[1]{%
  \csgdef{address@\the\c@author}{#1}}
\AtEndDocument{%
  \xdef\author@count{\the\c@author}%
  \c@author=1
  \print@authors}
\newcommand*\print@authors{%
  \ifnum\c@author>\author@count
  \else
    \print@author{\the\c@author}%
    \advance\c@author by 1
    \expandafter\print@authors
  \fi}
\newcommand*\print@author[1]{%
  \par\medskip
  \begin{tabular}{@{}l@{}}%
    \textsc{Addresses:}\\
    \csuse{address@#1}
  \end{tabular}}
\makeatother

\title{On a Galois property of fields generated by the torsion of an abelian variety}
\author{S. Checcoli and G. A. Dill}
\address{S. Checcoli: \\ Univ. Grenoble Alpes, CNRS, IF, 38000 Grenoble, France. \\\href{mailto: sara.checcoli@univ-grenoble-alpes.fr}{sara.checcoli@univ-grenoble-alpes.fr}\\
\\ G. A. Dill: \\ Leibniz Universit\"at Hannover,\\
Institut f\"ur Algebra, Zahlentheorie und Diskrete Mathematik,\\
Welfengarten 1,\\
30167 Hannover,\\
Germany.\\
\href{mailto: dill@math.uni-hannover.de}{dill@math.uni-hannover.de}\\ \\
Mathematisches Institut der Universit\"at Bonn,\\
Endenicher Allee 60,\\
53115 Bonn,\\
Germany.\\
\href{mailto: dill@math.uni-bonn.de}{dill@math.uni-bonn.de}}
\date{\today}

\newcommand\extrafootertext[1]{%
    \bgroup
    \renewcommand\thefootnote{\fnsymbol{footnote}}%
    \renewcommand\thempfootnote{\fnsymbol{mpfootnote}}%
    \footnotetext[0]{#1}%
    \egroup
}

\begin{document}
\maketitle
\extrafootertext{2020 \textit{Mathematics Subject Classification}. 11J95, 11R32.}
\abstract{In this article, we study a certain Galois property of subextensions of $k(A_{\mathrm{tors}})$, the minimal field of definition of all torsion points of an abelian variety $A$ defined over a number field $k$. Concretely, we show that each subfield of $k(A_{\mathrm{tors}})$ which is Galois over $k$ (of possibly infinite degree) and whose Galois group has finite exponent is contained in an abelian extension of some finite extension of $k$. As an immediate corollary of this result and a theorem of Bombieri and Zannier, we deduce that each such field has the Northcott property, \emph{i.e.} does not contain any infinite set of algebraic numbers of bounded height.}

\section{Introduction}
Throughout this article, we let $\Qbar$ be a fixed, once and for all, algebraic closure of $\Q$ and all algebraic extensions of $\mathbb{Q}$ that appear will be assumed to be subfields of $\Qbar$.

Let $G$ be an algebraic group defined over a number field $k$. Let $k(G_{\mathrm{tors}})$ denote the smallest subfield of $\Qbar$ containing $k$ over which all torsion points of $G$ are defined. For a number field $M$, we let $M^{\mathrm{ab}}$  denote the maximal abelian extension of $M$. Recall that a group has exponent $e$ if all its elements are torsion and the least common multiple of their orders equals $e$.
The goal of the present article is to prove the following theorem:
\begin{theorem}\label{mainav}
Let $A$ be an abelian variety defined over a number field $k$.
Let $e\geq 1$ be an integer. Then there exists a finite extension $M$ of $k$ that depends only on $A$, $k$, and $e$ such that for any subfield $F\subseteq k(A_{\mathrm{tors}})$ with $F/k$ Galois and $\Gal(F/k)$ of exponent at most $e$, one has $F\subseteq M^{\mathrm{ab}}$.
\end{theorem}
If, instead of $A$,  we consider $G=\mathbb{G}_m$, the multiplicative torus defined over a number field $k$, the analogue of Theorem \ref{mainav} is trivial (indeed, in this case, $k(G_{\mathrm{tors}})\subseteq k^{\mathrm{ab}}$). More interestingly, it can be proved, combining the seminal result \cite[Theorem 1]{BZ} with \cite[Proposition 2.1]{CZ}, that every extension of $k$ contained in $k(G_{\mathrm{tors}})$ (and more generally in $k^{\mathrm{ab}}$) and having Galois group of finite exponent satisfies the so-called Northcott property (N) introduced by Bombieri and Zannier in \cite{BZ}. More precisely, a subfield of $\Qbar$ is said to have \emph{property (N)} if it does not contain any infinite set of algebraic numbers of bounded absolute logarithmic Weil height. 

While, by Northcott's theorem (see for instance  \cite[Theorem 1.6.8]{BoGu}), property (N) is clearly satisfied by number fields, determining its validity for infinite extensions of $\Q$ is generally a hard problem. The fields described above are, thanks to \cite[Theorem 1]{BZ}, the first (and among the few) known examples of fields with this property.

One of our main motivations for this article was to prove an analogous result for abelian varieties. This is provided by the following corollary:
\begin{cor}\label{cor-main} Let $G=\mathbb{G}_m$ or $G=A$ with $A$ an abelian variety defined over a number field $k$. Let $k(G_{\mathrm{tors}})$ denote the smallest subfield of $\Qbar$ containing $k$ over which every torsion point of $G$ is defined. Then every subfield of $k(G_{\mathrm{tors}})$ which is Galois over $k$ and whose Galois group has finite exponent has property (N).
\end{cor}

The proof of Theorem \ref{mainav} is carried out in Section \ref{Sec-proof-main}. 
We remark that, unlike the toric case, the extension $k(A_{\mathrm{tors}})/k$ is highly non-abelian and our proof strategy relies on several important properties of the Galois representations associated with the Tate module of $A$ (summarized in the key Proposition \ref{properties-gal-rap}).

To briefly outline the main steps of the proof, we first notice that it suffices to prove Theorem \ref{mainav} if the extension $F/k$ is finite and therefore contained in some $k(A[n])$ where $A[n]$ denotes the kernel of multiplication by $n$ on $A(\Qbar)$. In Subsection \ref{Subsec-reduction-steps}, we show furthermore how to reduce the proof of Theorem \ref{mainav} to the construction of two subgroups $W_n'$ and $W_n''$ of $\mathrm{Gal}(k(A[n])/k)$ satisfying certain conditions. As we recall in Proposition \ref{properties-gal-rap} in Section \ref{prelim}, up to replacing $k$ by a fixed finite extension, we can write $\mathrm{Gal}(k(A[n])/k)$ canonically as the product of the Galois groups $\mathrm{Gal}(k(A[p^t])/k)$ for $p^t$ exactly dividing $n$. Moreover, for each prime $p$, the Galois representation on the $p$-adic Tate module of $A$ yields an algebraic group $\mathcal{G}_p$ over $\mathbb{Z}_p$, which has nice properties for all primes larger than some constant $c$. We then construct $W_n'$ and $W_n''$ in such a way that we can check the required conditions working ``prime by prime". For small primes $p$ (\emph{i.e.} those less than or equal to $\max\{c,e\}$), we use the $p$-adic exponential map associated to $\mathcal{G}_p$ while for large primes $p$, we use a group-theoretic property of $\mathcal{G}_p(\mathbb{F}_p)$ (see Proposition \ref{prop-MO} in Section \ref{prelim}).

In Section \ref{rem-Hab}, which is inspired by personal communication with Habegger, we discuss the validity of Theorem \ref{mainav} over other base fields: in particular, we show that, while it holds if $k$ is a finite field or a $p$-adic field, there are infinite extensions of $\Q$ over which Theorem \ref{mainav} is false. We remark that recently, motivated by the results of the present paper,  Gajda and Petersen  proved in \cite{GP} that Theorem \ref{mainav} holds over any field $k$ that is finitely generated over either its prime field, an algebraically closed field, or a local field.  
Finally, Corollary  \ref{cor-main} is proved in Section \ref{proof-cor} and is an easy consequence of our Theorem \ref{mainav}, \cite[Theorem 1]{BZ}, and  \cite[Proposition 2.1]{CZ}.

\section{Preliminaries on abelian varieties}\label{prelim}

Let $A$ be an abelian variety defined over a number field $K \subseteq \Qbar$.

We fix some notation. Let $K_n=K(A[n])$ be the smallest subfield of $\Qbar$ containing $K$ over which every point of $A[n]$, the group of torsion points of $A$ of order dividing $n$, is defined. Then the extension $K_n/K$ is finite and Galois. Set $g = \dim A$ and $G_n=\Gal(K_n/K)$, which we identify with a subgroup of $\GL_{2g}(\Z/n\Z)$ after having fixed once and for all a compatible system of isomorphisms $A[n] \simeq (\Z/n\Z)^{2g}$. 

For a prime number $p$, we consider the representation \[\rho_p: \Gal(\Qbar/K) \to \GL_{2g}(\Z_p),\] coming from the $p$-adic Tate module
\[ \varprojlim{A[p^n]} \simeq \varprojlim{(\Z/p^n\Z)^{2g}} \simeq \Z_p^{2g} \]
of $A$. We let $\mathcal{G}_p$ denote the Zariski closure of $\rho_p(\Gal(\Qbar/K))$ 
in $\mathrm{GL}_{2g,\Z_p}$ endowed with the unique structure of a reduced closed subscheme of $\mathrm{GL}_{2g,\Z_p}$.

Set now $K_{p^{\infty}} = \bigcup_{n \geq 0}{K_{p^n}}$, it is a Galois extension of $K$. Let $G_{p^\infty}$ denote its Galois group over $K$. Note that $\rho_p$ factors through the restriction homomorphism $\Gal(\Qbar/K) \to G_{p^{\infty}}$. Furthermore, the induced homomorphism $G_{p^{\infty}} \to \GL_{2g}(\Z_p)$ is injective and we identify $G_{p^{\infty}}$ with its image in $\GL_{2g}(\Z_p)$.

We collect here some well-known important classical results from the theory of Galois representations attached to abelian varieties over number fields. Recall that a smooth algebraic group $G$ over $\mathbb{Z}_p$ is called \emph{reductive} if both $G \otimes_{\mathbb{Z}_p} \mathbb{Q}_p$ and $G \otimes_{\mathbb{Z}_p} \mathbb{F}_p$ are reductive algebraic groups in the usual sense.
\begin{proposition}\label{properties-gal-rap} Let $A$ be an abelian variety defined over a number field $k$.
There exist a finite extension $K/k$ and a constant $c = c(A,K)>0$ such that, with notation as above, the following properties hold true:
\begin{enumerate}[(i)]
\item\label{property:directproduct}  for any coprime natural numbers $n'$ and $n''$, the canonical homomorphism $G_{n'n''} \to G_{n'} \times G_{n''}$ is an isomorphism,
\item\label{property:connected} the algebraic group $\mathcal{G}_p$ over $\Z_p$ is connected for every prime $p$,
\item\label{property:smoothandreductive} $\mathcal{G}_p$ is smooth and reductive for every prime $p > c$, and
\item\label{property:index} for every prime $p$, the index $[\mathcal{G}_p(\Z_p):\rho_p(\Gal(\Qbar/K))]$ is bounded from above by $c$. 
\end{enumerate}
\end{proposition}
\begin{proof}
Property \eqref{property:directproduct} follows from  {\cite[Th\'eor\`eme 1 and \S 3.1]{Ser13}} (see also \cite{Ser86}), while \eqref{property:connected} is a result from \cite{Ser81}.
Property \eqref{property:smoothandreductive} follows from \cite[Th\'eor\`eme 1 and \S 2.1]{Win} while the proof of \eqref{property:index} is detailed in \cite[Theorem 1.2 (a), Remark 1.3, and \S 6]{Zyw} and is based on, amongst others, results in \cite{Win} and \cite{Ser86}.
\end{proof}

\begin{remark}\label{rem-gal-rap}
Our proof does not use any properties of abelian varieties over number fields other than that Proposition \ref{properties-gal-rap} holds for the associated $p$-adic Galois representations: we can start with any system of continuous Galois representations $\rho_p: \mathrm{Gal}(\Qbar/K) \to \GL_{g(p)}(\Z_p)$ ($p$ prime, $g(p) \in \mathbb{N} = \{1,2,\hdots\}$) and define $K_n$ as the fixed field of $\bigcap_{i=1}^{r}{\ker (\pi_{p_i,k_i} \circ \rho_{p_i})}$ where $n = \prod_{i=1}^{r}{p_i^{k_i}}$ is the prime factorization of $n \in \mathbb{N}$ and $\pi_{p_i,k_i}: \GL_{g(p_i)}(\Z_{p_i}) \to \GL_{g(p_i)}(\Z/p_i^{k_i}\Z)$ is the canonical projection ($i = 1,\hdots,r$). Then, if the analogue of Proposition \ref{properties-gal-rap} holds in this situation, our proof will go through and show that any Galois subextension of $\bigcup_{n=1}^{\infty}{K_n}$ whose Galois group has finite exponent is contained in the maximal abelian extension of some number field.
\end{remark}

We will need the following fact, which is certainly well-known to experts. We give a proof which is heavily modeled after two answers to a question on MathOverflow (see \cite{MO, MO2}) and uses, for instance, classical results in \cite{BT, Lan, Steinberg}.

\begin{proposition}\label{prop-MO} Let $q$ be a power of a prime $p$ and let $G$ be a smooth, connected, reductive algebraic group over the finite field $\mathbb{F}_q$. Let $G(\mathbb{F}_q)^{+}$ be the normal subgroup of $G(\mathbb{F}_q)$ generated by its $p$-Sylow subgroups. Then the quotient $G(\mathbb{F}_{q})/G(\mathbb{F}_{q})^{+}$ is abelian.
\end{proposition}
\begin{proof}
Let $H$ be an arbitrary smooth, connected, reductive algebraic group over $\mathbb{F}_q$. We define $H(\mathbb{F}_q)^+$ analogously to $G(\mathbb{F}_q)^+$. By \cite[Theorem 14.5 and Proposition 14.14]{Milne}, a subgroup $\Gamma$ of $H(\mathbb{F}_q)$ is equal to the $\mathbb{F}_q$-points of a unipotent algebraic subgroup $U$ of $H$ if and only if the order of $\Gamma$ is a power of $p$. By \cite[Corollaire 3.7]{BT}, $H(\mathbb{F}_q)^{+}$ is then the subgroup of $H(\mathbb{F}_{q})$ generated by the $U(\mathbb{F}_q)$'s, where $U$ varies through the unipotent radicals of the minimal parabolic subgroups of $H$ that are defined over $\mathbb{F}_{q}$. Note that the minimal parabolic subgroups of $H$ are precisely the Borel subgroups of $H$ by \cite[Aside 17.73 and Proposition 17.99]{Milne}, where the latter proposition is a consequence of the main result in \cite{Lan}. In particular, $H$ has a Borel subgroup. Note that $k$ should be assumed to be finite in \cite[Proposition 17.99]{Milne} and the morphism $\sigma$ that occurs in the proof of \cite[Proposition 17.99]{Milne} is the Frobenius endomorphism relative to $k$.

Let $G'$ denote the derived subgroup of $G$, which is a semisimple algebraic group over $\mathbb{F}_q$, and let $\omega: \widetilde{G}' \to G'$ denote the universal cover of $G'$ (as an algebraic group; for the existence, see \cite[Remark 18.27]{Milne}). Then $\widetilde{G}'$ is also semisimple by \cite[Lemma 19.14]{Milne}. Since $G/G'$ is a torus by \cite[Proposition 12.46 and Corollary 21.50]{Milne}, the preceding paragraph together with \cite[Corollary 14.18]{Milne} implies that $G(\mathbb{F}_q)^{+} = G'(\mathbb{F}_q)^+$. By \cite[Lemma 2.5]{Iversen}, the preimage of a Borel subgroup $B$ of $G'$ under $\omega$ is a Borel subgroup $\widetilde{B}$ of $\widetilde{G}'$ (Iversen works over an algebraically closed field, but we can use \cite[17.66]{Milne} to descend to $\mathbb{F}_q$). Let $\widetilde{U}$ denote the unipotent radical of $\widetilde{B}$ and let $U$ denote the unipotent radical of $B$. We have that $\widetilde{U} \subseteq \omega^{-1}(U)$ since $B/U$ is of multiplicative type by \cite[Theorem 16.33]{Milne}. Thus, $\omega^{-1}(U)/\widetilde{U} \subseteq \widetilde{B}/\widetilde{U}$ is of multiplicative type while $\omega^{-1}(U)/\ker \omega \simeq U$ is unipotent. By \cite[Lemma 16.44]{Milne}, the multiplication morphism $(\ker \omega) \times \widetilde{U} \to \omega^{-1}(U)$ is an isomorphism. It follows that $\omega|_{\widetilde{U}}: \widetilde{U} \to U$ is an isomorphism and so $U(\mathbb{F}_q) = \omega(\widetilde{U}(\mathbb{F}_q))$. This implies that the map $\widetilde{G}'(\mathbb{F}_q)^+ \to G'(\mathbb{F}_q)^+$ is surjective. Thus, it suffices to show that $G(\mathbb{F}_q)/\omega(\widetilde{G}'(\mathbb{F}_q)^+)$ is abelian.

For an algebraic group $H$ over $\mathbb{F}_q$, we denote its center by $Z(H)$. The commutator morphism $G \times G \to G$ factors through the projection $G\times G \to (G/Z(G)) \times (G/Z(G))$. At the same time, we have a chain of canonical isomorphisms $\widetilde{G}'/Z(\widetilde{G}') \simeq G'/Z(G') \simeq G/Z(G)$, the first of which is induced by $\omega$, by \cite[Example 19.25 and Remark 19.30]{Milne} (the universal covering of $G'^{\mathrm{ad}} := G'/Z(G')$ is the composition of the universal covering of $G'$ with the projection $G' \to G'^{\mathrm{ad}}$ because of \cite[Proposition 18.2]{Milne}). This implies that the commutator morphism $G \times G \to G$ decomposes as the induced morphism $G \times G \to (\widetilde{G}'/Z(\widetilde{G}')) \times (\widetilde{G}'/Z(\widetilde{G}'))$ followed by first the morphism $(\widetilde{G}'/Z(\widetilde{G}')) \times (\widetilde{G}'/Z(\widetilde{G}')) \to \widetilde{G}'$ that is induced by the commutator morphism of $\widetilde{G}'$ and then $\omega$. It follows that $\omega(\widetilde{G}'(\mathbb{F}_q))$ contains the commutator subgroup of $G(\mathbb{F}_q)$.

Combining the two preceding paragraphs shows that it suffices to prove that $\widetilde{G}'(\mathbb{F}_q)^+ = \widetilde{G}'(\mathbb{F}_q)$. 

By \cite[Remark 18.27]{Milne},  $\widetilde{G}'$ is still simply connected when base changed to an algebraic closure $\bar{\mathbb{F}}_q$ of $\mathbb{F}_q$ and, by \cite[Proposition 23.59]{Milne}, over $\bar{\mathbb{F}}_q$, the definitions of simply connectedness in \cite[Definition 18.5]{Milne} and \cite[\S 6.4]{Steinberg} are equivalent. We can then apply \cite[Theorem 12.4]{Steinberg} to $\widetilde{G}'$ over $\bar{\mathbb{F}}_q$ with $\sigma$ chosen to be the Frobenius endomorphism relative to $\mathbb{F}_q$ to conclude.
\end{proof}

\section{Proof of Theorem \ref{mainav}}\label{Sec-proof-main}

\subsection{Some reduction steps}\label{Subsec-reduction-steps}

Let $F/k$ be a (possibly infinite) Galois extension with $F\subseteq k(A_{\mathrm{tors}})$ and $\Gal(F/k)$ having finite exponent $\exp(\Gal(F/k))\leq e$ for some positive integer $e\geq 1$. As $F$ equals the compositum of all its subextensions which are finite and Galois over $k$, it is enough to prove that Theorem \ref{mainav} holds true when $[F:k]$ is finite. 

So, let $F/k$ be finite and Galois with $\Gal(F/k)$ of exponent at most $e$ and suppose $F\subseteq \kAtor$.

Let $K/k$  be the finite extension from Proposition \ref{properties-gal-rap} and let $L=KF \subseteq \KAtor$ be the compositum of $F$ and $K$. Note that
\[\Gal(L/K)\simeq \Gal(F/F\cap K)\leq \Gal(F/k),\] hence $\exp(\Gal(L/K))\leq e$.

We are going to prove that  $L$ is contained in the maximal abelian extension of a fixed number field containing $K$ and depending only on $A$, $K$, and $e$. This of course will prove Theorem \ref{mainav} for the extension $F/k$.

As before, we denote by $K_n=K(A[n])$ the smallest subfield of $\Qbar$ containing $K$ over which every point of $A[n]$, the group of all torsion points of $A$ of order dividing $n$, is defined. We set $g = \dim A$ and $G_n=\Gal(K_n/K)$, which we identify with a subgroup of $\GL_{2g}(\Z/n\Z)$, having fixed once and for all a compatible system of isomorphisms $A[n] \simeq (\Z/n\Z)^{2g}$. 

We fix an integer $n$ such that $L \subseteq K_n$  and we set $H_n=\Gal(K_n/L)\leq G_n$.
Our goal is to prove that there exists a positive integer $m = m(A,K,e)$ and subgroups $W'_n$ and $W''_n$ of $G_n$ such that:   
\begin{enumerate}[(a)]
\item\label{cond-quot-ab} $W''_n\unlhd G_{n}$ and $G_{n}/W''_n$ is abelian,
\item\label{cond-proj-m} $W'_n$ contains the kernel of the projection from $G_{n}$ onto $G_{\gcd(n,m)}$, and
\item\label{cond-inters} $W'_n \cap W''_n\subseteq H_n$.
\end{enumerate}
From these properties it follows that 
 \[ L = K_n^{H_n} \subseteq K_n^{W'_n \cap W''_n} = K_n^{W'_n}K_n^{W''_n}\]
where we use the usual Galois-theoretic notation for the fixed field of a subgroup of the Galois group. We conclude that
\[ L \subseteq K_mK^{\mathrm{\mathrm{ab}}} \subseteq K_m^{\mathrm{\mathrm{ab}}}\]
and the statement of Theorem \ref{mainav} holds with $M=K_m$.

\subsection{Construction of the auxiliary subgroups}
We let $c=c(A,K)$ be the constant from Proposition \ref{properties-gal-rap} and set \[\delta = e!c!.\]
In the following, $\mathbb{I}_{2g}$ will denote the $2g \times 2g$ identity matrix. By abuse of notation, we use the same symbol regardless of the ring of coefficients of the matrix.

We write $n = n'n''$ where $n'$ and $n''$ are the unique coprime positive integers satisfying that $p \mid n' $ if and only if $p\mid \gcd(n,\delta)$ for every prime $p$.

Writing $n'=\prod_{i=1}^r p_i^{a_i}$ and $n''=\prod_{j=1}^{s}{q_j^{b_j}}$ with $p_i$ and $q_j$ distinct primes and $a_i$ and $b_j$ positive integers ($i=1,\hdots,r$, $j=1,\hdots,s$), we can apply Proposition \ref{properties-gal-rap}.\eqref{property:directproduct} to identify 
\begin{equation}\label{dec-Gn}G_n=G_{n'}\times G_{n''} = \left(\prod_{i=1}^{r}{G_{p_i^{a_i}}}\right)\times \left(\prod_{j=1}^{s}{G_{q_j^{b_j}}}\right).\end{equation}

Consider, under this identification, the subgroups of $G_n$ given by  \[U'_n=H_n\cap\left(G_{n'} \times \{(\mathbb{I}_{2g},\hdots,\mathbb{I}_{2g})\}\right)\] and \[U''_n=H_n\cap\left(\{(\mathbb{I}_{2g},\hdots,\mathbb{I}_{2g})\} \times G_{n''}\right)\]
 and let $H'_n$ and $H''_n$ be their isomorphic images in $G_{n'}$ and $G_{n''}$ respectively.
Let finally  \[W'_n= H'_n\times G_{n''}\] and \[W''_n=G_{n'}\times H''_n.\]

We are going to prove that these groups satisfy properties \eqref{cond-quot-ab}, \eqref{cond-proj-m}, and \eqref{cond-inters} from above. This will conclude the proof of Theorem \ref{mainav}.

Note that, under the identification \eqref{dec-Gn}, one has that  the element
\[\gamma=(\gamma'_1,\ldots,\gamma'_r,\gamma''_1,\ldots,\gamma''_s)\]
of $G_n$ belongs to $W'_n\cap W''_n$ if and only if the two elements \[\gamma'=(\gamma'_1,\ldots,\gamma'_r,\mathbb{I}_{2g},\ldots,\mathbb{I}_{2g})\] and \[\gamma''=(\mathbb{I}_{2g},\ldots,\mathbb{I}_{2g},\gamma''_1,\ldots,\gamma''_s)\] are both in $H_n$. As clearly $\gamma=\gamma'\gamma''$, we get that $W'_n\cap W''_n\subseteq H_n$ and so \eqref{cond-inters}  holds.

\begin{claim}\label{lem:largeprimes} 
The group $W''_n$ satisfies condition \eqref{cond-quot-ab}.
\end{claim}

\begin{proof}[Proof of Claim \ref{lem:largeprimes}]
Since $L/K$ is Galois, the subgroup $H_n$ is normal in $G_n$, so $H''_n$ is normal in $G_{n''}$ and hence $W''_{n}$ is normal in $G_n$. Note that
\[G_n/W''_n\simeq G_{n''}/H''_n\simeq \left( \{(\mathbb{I}_{2g},\hdots,\mathbb{I}_{2g})\}\times G_{n''}\right)/U''_n\]
and we are going to prove that the last quotient is abelian.

 For $j = 1, \hdots, s$, set
\[U_{q_j^{b_j}} = U''_n \cap (\{(\mathbb{I}_{2g},\hdots,\mathbb{I}_{2g})\} \times G_{q_j^{b_j}} \times \{(\mathbb{I}_{2g},\hdots,\mathbb{I}_{2g})\}), \]
where of course the only non-trivial factor occurs at the $(r+j)$-th position, and let $H_{q_j^{b_j}}$ denote its isomorphic image in $G_{q_j^{b_j}}$. Using the identification \eqref{dec-Gn}, we have that
\[ \{(\mathbb{I}_{2g},\hdots,\mathbb{I}_{2g})\} \times \prod_{j=1}^{s}{H_{q_j^{b_j}}} \subseteq U''_n.\]
It follows that $ \left( \{(\mathbb{I}_{2g},\hdots,\mathbb{I}_{2g})\}\times G_{n''}\right)/U''_n$ is isomorphic to a quotient of
\[(\{(\mathbb{I}_{2g},\hdots,\mathbb{I}_{2g})\}\times \prod_{j=1}^{s}{G_{q_j^{b_j}}})/(\{(\mathbb{I}_{2g},\hdots,\mathbb{I}_{2g})\} \times \prod_{j=1}^{s}{H_{q_j^{b_j}}}),\]
which, in turn, is isomorphic to the product
\[\prod_{j=1}^{s}{(G_{q_j^{b_j}}/H_{q_j^{b_j}})}.\]
Hence it suffices to show that each factor $G_{q_j^{b_j}}/H_{q_j^{b_j}}$ is abelian for $j = 1,\hdots,s$. Note that this group has finite exponent at most $e$ since it admits an injective homomorphism into the group $G_n/H_n$ of finite exponent at most $e$.

From now on, we fix $j$ and we write $q,b$ instead of $q_j,b_j$ to ease the reading.

Let $\mathcal{G}_q$ be the algebraic group over $\Z_q$ defined in Section \ref{prelim} and
consider the homomorphism \[\varphi: \mathcal{G}_q(\Z/q^b\Z) \to \mathcal{G}_q(\Z/q\Z)\] given by reduction modulo $q$. 
We have that $\varphi$ is surjective by Proposition \ref{properties-gal-rap}.\eqref{property:smoothandreductive} and Hensel's lemma. 
Furthermore, the order of the kernel of $\varphi$ is a power of $q$.

Let $\mathcal{G}_q(\Z/q\Z)^+$ denote the normal subgroup of $\mathcal{G}_q(\Z/q\Z)$ that is generated by its $q$-Sylow subgroups. 

Note that $\mathcal{G}_q(\Z/q\Z)^+$ is generated by elements $g$ whose order is a power of $q$ and any lift $\gamma \in \mathcal{G}_q(\Z/q^b\Z)$ of such an element $g$ will have order equal to a power of $q$ as well. Proposition \ref{properties-gal-rap}.\eqref{property:index} together with Hensel's lemma implies that $\gamma^{c!}\in G_{q^b}$ for any such $\gamma$. 

Also, as $\exp(G_{q^b}/H_{q^b})\leq e$, $\gamma^{\delta}\in H_{q^b}$ and finally, as $\gamma$ has order equal to a power of $q$ and $q\nmid \delta$, we have that $\gamma\in H_{q^b}$. It follows that $\ker \varphi \subseteq H_{q^b}$ and $\mathcal{G}_q(\Z/q\Z)^+ \subseteq \varphi(H_{q^b})$. Hence, we have that $\varphi^{-1}(\mathcal{G}_q(\Z/q\Z)^+) \subseteq H_{q^b}$.

Thus $G_{q^b}/H_{q^b}$ is a quotient of the group $G_{q^b}/\varphi^{-1}(\mathcal{G}_q(\Z/q\Z)^+)$, which admits an injective homomorphism into the group \[\mathcal{G}_q(\Z/q\Z)/\mathcal{G}_q(\Z/q\Z)^+.\] This last group is abelian by Proposition \ref{prop-MO} and so is therefore $G_{q^b}/H_{q^b}$.
\end{proof}

\begin{claim}\label{lem:smallprimes}
The group $W'_{n}$ satisfies condition \eqref{cond-proj-m}.

\end{claim}

\begin{proof}[Proof of Claim \ref{lem:smallprimes}]
For $i = 1,\hdots,r$, set
\[U_{p_i^{a_i}} = U'_n \cap (\{(\mathbb{I}_{2g},\hdots,\mathbb{I}_{2g})\} \times G_{p_i^{a_i}} \times \{(\mathbb{I}_{2g},\hdots,\mathbb{I}_{2g})\}), \]
where now the only non-trivial factor occurs at the $i$-th position, and let $H_{p_i^{a_i}}$ denote its isomorphic image in $G_{p_i^{a_i}}$. 
Note that, as before, the quotient group $G_{p_i^{a_i}}/H_{p_i^{a_i}}$ has exponent at most $e$ since it admits an injective homomorphism into the group $G_n/H_n$ of exponent at most $e$.
Using the identification \eqref{dec-Gn}, we have
\[ \prod_{i=1}^{r}{H_{p_i^{a_i}}} \times G_{n''} \subseteq W'_n.\]
Thus, thanks to Proposition \ref{properties-gal-rap}.\eqref{property:directproduct}, in order to prove Claim \ref{lem:smallprimes} it suffices to show that, for each $i = 1,\hdots,r$, there exists some $t_i = t_i(A,K,e) \in \mathbb{N}$ such that $H_{p_i^{a_i}}$ contains the kernel of the projection map
\[\pi_{p_i^{t_i}}:G_{p_i^{a_i}} \rightarrow G_{p_i^{\min\{t_i,a_i\}}}.\]
 From now on, we fix $i$ and write $p, a, t$ instead of $p_i, a_i, t_i$.

Let $G_{p^\infty}=\Gal(K_{p^{\infty}}/K)\subseteq \mathcal{G}_p(\Z_p)$ be the group defined in Section \ref{prelim}.
Let \[\lambda: G_{p^\infty} \to G_{p^a}\] be the canonical surjective restriction homomorphism and set ${H}_{p^{\infty}}=\lambda^{-1}(H_{p^a})$. 

Note that it is enough to show that there exists some $t = t(A,K,e) \in \mathbb{N}$ such that ${H}_{p^{\infty}}$ contains the kernel of the projection map
\[\tilde{\pi}_{p^t}:\mathcal{G}_p(\Z_p) \to \mathcal{G}_p(\Z/p^t\Z).\]
Indeed this would imply that $\ker(\pi_{p^t})=\lambda(\ker(\tilde{\pi}_{p^t}))\subseteq H_{p^a}$.

As $H_{p^a}$ is normal in $G_{p^a}$, the subgroup ${H}_{p^{\infty}}$ is normal in $G_{p^\infty}$. Since $G_{p^\infty}/{H}_{p^{\infty}} \simeq G_{p^a}/H_{p^a}$ has finite exponent at most $e$ and \[[\mathcal{G}_p(\Z_p):G_{p^{\infty}}] = [\mathcal{G}_p(\Z_p):\rho_p(\Gal(\Qbar/K))]\leq c\] 
 by Proposition \ref{properties-gal-rap}.\eqref{property:index}, we have that, if $\gamma \in \mathcal{G}_p(\Z_p)$, then $\gamma^{\delta}$ lies in ${H}_{p^{\infty}}$.

We now want to show that $t$ can be chosen such that every element in $\ker(\tilde{\pi}_{p^t})$ is of the form $\gamma^{\delta}$ for some $\gamma\in  \mathcal{G}_p(\Z_p)$. This will complete the proof of Claim \ref{lem:smallprimes} and of Theorem \ref{mainav}.

Let $\mathbf{L}_p$ denote the Lie algebra of $\mathcal{G}_p \otimes_{\Z_p}\Q_p$.
 By \cite[Ch. III, \S 7, No. 2, Proposition 3]{Bour2},  there is an open subgroup $\mathcal{U} \subseteq \mathbf{L}_p$ and a map
 \[\varphi: \mathcal  U \to \mathcal{G}_p(\Q_p),\]
 called exponential map and satisfying $\varphi(\ell u) = \varphi(u)^{\ell}$ for all $u \in \mathcal{U}$ and $\ell \in \mathbb{Z}$ and such that $\mathcal V=\varphi(\mathcal{U})$ is an open subgroup of $\mathcal{G}_p(\Q_p)$ homeomorphic to  $\mathcal{U}$ via $\varphi$. After shrinking $\mathcal{U}$, we can and will assume without loss of generality that $\varphi(\mathcal{U}) \subseteq \mathcal{G}_p(\Z_p)$ since $\mathcal{G}_p(\Z_p)$ is an open subgroup of $\mathcal{G}_p(\Q_p)$.

Note that \[\delta \mathcal  U=\{\delta u\mid u\in \mathcal{U}\}\] is open in $\mathcal  U$ as multiplication by $\delta$ is an automorphism of $\mathbf{L}_p$ as a topological group.
Furthermore, since $\varphi$ is a homeomorphism and $\varphi(0) = \mathbb{I}_{2g}$, we have that $\varphi(\delta \mathcal  U)$ is open in $\mathcal  V$ and contains $\mathbb{I}_{2g}$.

Therefore there exists $t = t(A,K,e) \in \mathbb{N}$ such that $\varphi(\delta \mathcal  U)$ contains the open ball (with respect to the maximum norm on $\mathcal{G}_p(\Z_p) \subseteq \GL_{2g}(\Z_p) \subseteq \Q_p^{4g^2}$) centered at $\mathbb{I}_{2g}$ of radius $p^{-t+1}$. Thus, $\ker(\tilde{\pi}_{p^t}) \subseteq \varphi(\delta \mathcal  U) \subseteq \mathcal  V$.

We set $\psi = \varphi^{-1}: \mathcal{V} \to \mathcal{U}$. Then, if $v \in \ker(\tilde{\pi}_{p^t})$, we have that $\psi(v) \in \delta \mathcal  U$ and so $\psi(v) = \delta u$ for some $u \in \mathcal  U$. Hence
\[ v = \varphi(\psi(v)) = \varphi(\delta u) = \varphi(u)^{\delta} =\gamma^{\delta}\]
with $\gamma=\varphi(u)\in \mathcal{G}_p(\Z_p)$ and we are done. 
\end{proof}

\section{Theorem \ref{mainav} over other base fields}\label{rem-Hab}
In a personal communication, P. Habegger asked two questions: the first was whether Theorem \ref{mainav} holds over other base fields such as global fields of positive characteristic or local fields; the second was about the existence of fields $k$ for which Theorem \ref{mainav} does not hold.

Concerning the first question, if $k$ is a finite field, then Theorem \ref{mainav} trivially holds true with $M=k$ as any algebraic extension of $k$ is abelian. 

If $k$ is a $p$-adic field, then Theorem \ref{mainav} follows from the argument in the proof of \cite[Theorem 1.2]{C}, which we re-sketch here for clarity.
Suppose that $F/k$ is finite and Galois and $\mathrm{Gal}(F/k)$ has exponent at most $e$.
First notice that the Galois group of the tamely ramified part $F^{\mathrm{tame}}/k$ of the extension $F/k$ is metacyclic and hence has order bounded by $e^2$. It then follows from the proof of \cite[Theorem 1]{Shafarevich} (which can be adapted to the case where $k$ contains the $p$-th roots of unity) that $\mathrm{Gal}(F/F^{\mathrm{tame}})$ is a finite $p$-group generated by at most $[k:\mathbb{Q}_p]e^2+2$ elements (and of exponent at most $e$), so, by Zelmanov's solution to the restricted Burnside problem (see \cite{Zel1, Zel2} and previous work in \cite{HH56, Kos59}), it has order bounded from above by a constant depending only on $k$ and $e$.
Since $\mathbb{Q}_p$ has only finitely many finite extensions of a given degree and any Galois extension of $k$ is the compositum of its finite Galois subextensions, Theorem \ref{mainav} therefore also holds true in this case and for reasons unrelated to abelian varieties.

Motivated by our work, in \cite{GP} Gajda and Petersen  proved, among other results, that Theorem \ref{mainav} holds over any field $k$ that is finitely generated over either its prime field, an algebraically closed field, or a local field. Their approach is different and more of topological nature, based, in particular, on work of Zelmanov and Wilson.

As for the second question posed by Habegger, it is possible to construct an infinite extension of $\mathbb{Q}$ for which Theorem \ref{mainav} is false as follows: we choose an abelian surface $A$  defined over a number field $k$ with (geometric) endomorphism ring equal to $\mathbb{Z}$. As noticed in Section \ref{prelim}, after replacing $k$ with a finite extension $K/k$, we have that the $p$-adic  representations $\rho_p: \Gal(\Qbar/K) \to \GL_{4}(\Z_p)$ (defined in Section \ref{prelim}) are independent of each other (\emph{i.e.}, the image of the adelic representation is precisely the product of the images of the $p$-adic representations).
Furthermore, it is proved in \cite[Theorem 1.3]{Lom} that, for a suitably chosen system of isomorphisms $A[n] \simeq (\mathbb{Z}/n\mathbb{Z})^4$ (cf. \cite[Section 2.1]{Lom}), we have that $\mathrm{Im}(\rho_p)= \mathrm{GSp}_4(\mathbb{Z}_p)$ for all sufficiently large $p$.

We deduce from this together with Dirichlet's theorem on primes in arithmetic progressions that the set $\mathcal{P}$ of rational primes $p$ such that $p \equiv 1 \mod 5$ and $\mathrm{Im}(\rho_p) = \mathrm{GSp}_4(\mathbb{Z}_p)$ is infinite. For $p \in \mathcal{P}$, let $\zeta \in \mathbb{F}_p$ denote a primitive fifth root of unity, then the subgroup
\[ \left\langle  \begin{pmatrix} \zeta & 1 \\ 0 & \zeta^{-1} \end{pmatrix}, \begin{pmatrix} -1 & -\zeta^{-1} \\\zeta & 0 \end{pmatrix} \right\rangle \]
of $\mathrm{SL}_2(\mathbb{F}_p)$ is isomorphic to $\mathrm{SL}_2(\mathbb{F}_5)$ (see \cite{MS} and \cite[Proposition 13.7]{Passman}). Furthermore, there is an injective group homomorphism from $\mathrm{SL}_2(\mathbb{F}_p)$ to $\mathrm{GSp}_4(\mathbb{F}_p)$ that sends a matrix $M$ to \[\begin{pmatrix} M & 0 \\ 0 & M^{-T}\end{pmatrix}\] where $M^{-T}$ denotes the inverse transpose of $M$.

It follows that for all $p\in \mathcal{P}$, the group $\mathrm{Im}(\rho_p)=\mathrm{GSp}_4(\mathbb{F}_p)$ contains a non-abelian finite subgroup $G_p \simeq \mathrm{SL}_2(\mathbb{F}_5)$ of fixed order. 

For $M \in \mathrm{GSp}_4(\mathbb{Z}_p)$, we let $\pi_p(M)$ denote its reduction modulo $p$. Since the $\rho_p$ are continuous and independent of each other, the product of the groups
\[\{M  \in \mathrm{GSp}_4(\mathbb{Z}_p)\,\mid\, \pi_p(M) \in G_p\}\]
for $p$ varying in $\mathcal{P}$ is identified with a closed subgroup $\Gamma$ of $\mathrm{Gal}(K(A_{\mathrm{tors}})/K)$ via the adelic representation. By Galois theory, there exists an (infinite) extension $K'/K$ contained in $K(A_{\mathrm{tors}})$ such that $\mathrm{Gal}(K(A_{\mathrm{tors}})/K')$ is equal to $\Gamma$. In the same way, the Galois correspondence also provides a Galois extension $F/K'$ with $F \subseteq K(A_{\mathrm{tors}})$ such that $\mathrm{Gal}(F/K')$ is isomorphic to the product of the groups $G_p$ for $p\in \mathcal{P}$. This product has finite exponent, but does not contain any abelian subgroup of finite index.

\section{Proof of Corollary  \ref{cor-main}}\label{proof-cor}
For a number field $k$ and a positive integer $e$, let $k^{(e)}$ denote the compositum of all extensions of $k$ of degree at most $e$ and let $k^{(e)}_{\mathrm{ab}}$ be the maximal abelian extension of $k$ contained in $k^{(e)}$.

We recall that, from the proof of \cite[Proposition 2.1]{CZ}, it follows that any abelian extension of $\Q$ having Galois group of exponent at most $e$ is contained in $\Q^{(e)}_{\mathrm{ab}}$. We repeat that proof here, the only difference being that the base field $\Q$ is replaced by any number field $L$. 

Let $F/L$ be a possibly infinite abelian extension of a number field $L$ and suppose that $G=\text{Gal}(F/L)$ has exponent at most $e$. Clearly, $F$ equals the compositum of all its Galois subextensions $L\subseteq F'\subseteq F$ with 
$G'=\text{Gal}(F'/L)$ finite and
abelian of exponent at most $e$. We can thus write $G'=\prod_{i=1}^{n}U_i$ as a direct product of finite cyclic groups $U_i$ of order at most $e$.
Let $H_i$ be the subgroup of $G'$ defined by $H_i:=\prod_{j\neq
i}U_j$. We have that $[G':H_i]=|U_i|\leq e$ for all $i$'s and $\cap_{i=1}^n H_i=\{\mathrm{id}\}$.
Thus, $F'$ equals the compositum of the fields $F'^{H_{1}},\ldots,F'^{H_{n}}$,
which are abelian extensions of $L$ of degrees $[F'^{H_{i}}:L]=[G':H_i]\leq e$. Hence $F'\subseteq L^{(e)}_{\mathrm{ab}}$ for every $F'$ as above,  which implies that $F\subseteq
L^{(e)}_{\mathrm{ab}}$.

Now let $F$ be a subfield of $k(G_{\mathrm{tors}})$ such that $F/k$ is Galois and $\mathrm{Gal}(F/k)$ has exponent at most $e$.

If $G=\mathbb{G}_m$, then $\mathrm{Gal}(F/k)$ is abelian and, by the above-mentioned results, $F\subseteq k^{(e)}_{\mathrm{ab}}$. 

If $G=A$, then, by Theorem \ref{mainav}, there exists a finite extension $M/k$, depending only on $A$, $k$, and $e$, such that $F\subseteq M^{\mathrm{ab}}$.  
Setting $L=F\cap M$, we have that $\Gal(F/L)$ is abelian and has exponent at most $e$ as a subgroup of $\Gal(F/k)$. So again $F\subseteq L^{(e)}_{\mathrm{ab}}$. 

We now conclude the proof using the fact that fields of the form $L^{(e)}_{\mathrm{ab}}$ have property (N) by \cite[Theorem 1]{BZ}.

\section*{Acknowledgements}

The authors thank Michel Brion, Philipp Habegger, Samuel Le Fourn, David Loeffler, Davide Lombardo, Thomas Scanlon, and Umberto Zannier  for useful discussions and correspondence and for their kind interest in this work.

This article is mostly based upon work supported by the National Science Foundation under Grant Nos. 1440140 and DMS-1928930 while the authors were in residence at the Mathematical Sciences Research Institute in Berkeley, California, during the Spring 2023 semester. We thank the organisers of the ``Diophantine Geometry" program for giving us the opportunity to participate in this program and the MSRI staff for providing an excellent environment for mathematical collaborations.

Sara Checcoli's work has been also supported by the French National Research Agency in the framework of the \textit{Investissements d'avenir} program (ANR-15-IDEX-02, IRGA project \emph{ProBoNo}). Gabriel Dill has received funding from the European Research Council (ERC) under the European Union's Horizon 2020 research and innovation programme (grant agreement n$^\circ$ 945714).

\end{document}